\documentclass[a4paper]{article}
\usepackage{amsmath,amssymb,amsthm,amsfonts}
\usepackage{enumerate,color, graphicx}
\newtheorem{theorem}{Theorem}[section]

\newtheorem{proposition}[theorem]{Proposition}

\newtheorem{conjecture}[theorem]{Conjecture}

\newtheorem{remark}[theorem]{Remark}

\def \rz {{\mathbb R}}

\def \rz {{\mathbb R}}

\newcommand {\ar}{\rightarrow}
\newcommand {\pa}{\partial}

\newcommand{\ep}{\epsilon}

\numberwithin{equation}{section}
\begin{document}
{\centering
\bfseries
{\Large On the multiplicity of the second eigenvalue of the Laplacian in non simply connected domains\\--with some numerics--}
\\
\noindent
\par
\mdseries
\scshape
\small
B. Helffer$^{*,**}$\\
T. Hoffmann-Ostenhof$^{***}$ \\
F. Jauberteau$^{*}$\\
C. L\'ena $^{****}$\\
\par
\upshape
Laboratoire de Math\'ematique Jean Leray, Univ. Nantes $^*$\\
Laboratoire de Math\'ematiques d'Orsay, Univ Paris-Sud and CNRS $^{**}$\\
Institut f\"ur Theoretische Chemie, Universit\"at Wien $^{***}$\\
Department of Mathematics, Stockholm University $^{****}$\\~\\

}

\begin{abstract} We revisit an interesting example proposed by Maria Hoffmann-Ostenhof,  the second author and Nikolai Nadirashvili of a bounded domain in $\mathbb R^2$ for which the second eigenvalue of the Dirichlet Laplacian has multiplicity $3$.
We also analyze  carefully the first eigenvalues of the Laplacian in the case of the disk with two symmetric cracks  placed on a smaller concentric disk in function of their size.
\end{abstract}

\newpage

\section{Introduction}
The motivating problem is to analyze the multiplicity of the $k$-th eigenvalue of the Dirichlet problem in a domain $\Omega$  in $\mathbb R^2$. It is for example an old result of Cheng \cite{Ch}, that the multiplicity of the second eigenvalue is at most $3$. 
In \cite{HOHON2} an example  with multiplicity $3$ is given  as a side product of the production of a counter example to the nodal line conjecture (see also \cite{HOHON1},
  and the papers by Fournais  \cite{Fo} and
Kennedy \cite{K}  who extend to higher dimensions  these counter
examples, introducing new methods). This example is based on the spectral analysis of the Laplacian in domains consisting of a disc in which we have introduced on an interior concentric circle suitable  cracks.\\  We discuss the initial proof and complete it by one missing argument. For completion, we will also extend the validity of a theorem of Cheng to less regular domains.
\begin{figure}[!ht]\label{fig1}
 \begin{center}
\includegraphics[width=10cm]{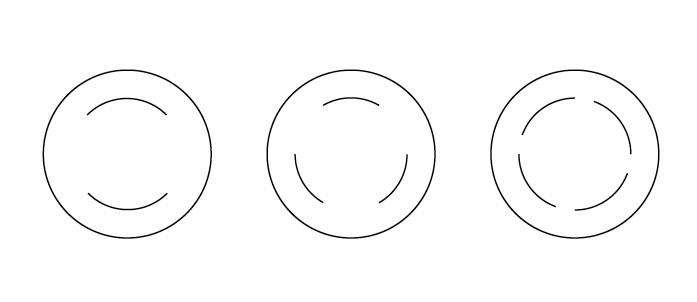}
 \caption{The domains with cracks for $N=2$, $N=3$ and $N=4$.} 
  \end{center}
 \end{figure}
 Although not needed for the positive results, we complete the paper with numerical results illustrating why some argument has to be modified 
   and propose  a fine theoretical analysis of the spectral problem when the cracks are closed.
 \section{Main statement}
 The starting point for the construction of counterexamples  to the nodal line conjecture \cite{HOHON1,HOHON2}  is the introduction of two concentric open discs $B_{R_1}$ and $B_{R_2}$ with $0 < R_1 < R_2$ and  the  corresponding annulus $M_{R_1,R_2} = B_{R_2} \setminus \bar B_{R_1}$.
 The authors choose  $R_1$ and $R_2$ such that
 \begin{equation}\label{ineqa}
 \lambda_1(B_{R_1}) < \lambda_1 ( M_{R_1,R_2} )  <   \lambda_2(B_{R_1})\,,
 \end{equation}
 where, for $\omega\subset \mathbb R^2$ bounded,  $\lambda_j(\omega)$ denotes the $j$-th eigenvalue of the Dirichlet Laplacian $H$ in $\omega$.\\
 We observe indeed that for fixed $R_1$,  $\lambda_1 ( M_{R_1,R_2} )$ tends to $+\infty$ as $R_2 \ar R_1$ (from above) and tends to $0$ as $R_2 \rightarrow +\infty\,$. Moreover $R_2 \mapsto \lambda_1 ( M_{R_1,R_2} )$ is decreasing. Hence there is some interval $(a(R_1),b(R_1))$  with $a(R_1) >R_1$ such that \eqref{ineqa} is satisfied if and only if  $R_2 \in (a(R_1),b(R_1))$.\\
 Then we  introduce $$ D_{R_1,R_2} = B_{R_1} \cup M_{R_1,R_2}$$
 and observe that
 \begin{equation}\label{specannulus}
 \begin{array}{l}
 \lambda_1(D_{R_1,R_2} )=\lambda_1(B_{R_1})\\
  \lambda_2(D_{R_1,R_2} )=  \lambda_1 ( M_{R_1,R_2} )\\
  \lambda_3 (D_{R_1,R_2} )= \min(\lambda_2(B_{R_1}), \lambda_2 ( M_{R_1,R_2} ))\,.
 \end{array}
 \end{equation}
 
 If Condition \eqref{ineqa} was important in the construction of the counter-example to the nodal line conjecture, the weaker assumption 
 \begin{equation}\label{ineqb}
 \max (\lambda_1(B_{R_1}),  \lambda_1 ( M_{R_1,R_2} )) < \min (\lambda_2(B_{R_1}), \lambda_2 ( M_{R_1,R_2} ))\,.
 \end{equation}
 suffices for the multiplicity question. Under this condition, we have:
 \begin{equation}\label{specannulusa}
 \begin{array}{l}
 \lambda_1(D_{R_1,R_2} )=\min (\lambda_1(B_{R_1}),  \lambda_1 ( M_{R_1,R_2} ))\\
  \lambda_2(D_{R_1,R_2} )=  \max (\lambda_1(B_{R_1}),  \lambda_1 ( M_{R_1,R_2} )) \\
  \lambda_3 (D_{R_1,R_2} )= \min (\lambda_2(B_{R_1}), \lambda_2 ( M_{R_1,R_2} ))\,,
 \end{array}
 \end{equation}
 and it is not excluded (we are in the non connected situation) to consider the case $ \lambda_1(D_{R_1,R_2} )= \lambda_2(D_{R_1,R_2} )$.

 We now carve holes in $\pa B_{R_1}$ such that $D_{R_1,R_2}$ becomes a domain. For $N\in \mathbb N^*:=\mathbb N \setminus \{0\}$ and $\epsilon \in [0,\frac  \pi N]$, we introduce  (see Figure $1$  for $N=2, 3$)
 \begin{equation}
 \mathfrak D(N,\epsilon)=D_{R_1,R_2} \cup_{j=0}^{N-1} \{x\in \mathbb R^2\,,\, r=R_1\,,\, \theta \in (\frac {2\pi j}{N}-\epsilon, \frac {2\pi j}{N} + \epsilon)\}\,.
 \end{equation}

 The theorem stated in \cite{HOHON2} is the following:
 \begin{theorem}\label{Theorem2.1}
 Let $N\geq 3$, then there exists $\epsilon \in (0, \frac  \pi N)$ such that $\lambda_2(\mathfrak D(N,\epsilon))$ has multiplicity $3$.
 \end{theorem}

 We  prove below that the theorem is correct. But the proof given  in \cite{HOHON2}  works only for even
integers $N\ge 4$ and in this case there is a need for additional arguments. So we improve in this paper the result in  \cite{HOHON2}  by giving
an example $\Omega:= \mathfrak D(3,\epsilon) $ where the number of components of $\pa \Omega $ equals~4,  hence $N=3$.

 \begin{remark}
Theorem \ref{Theorem2.1} leads to the
following question:\\
 Is there a bounded domain $\Omega\subset \rz^2$ whose boundary $\pa \Omega$ has strictly  less than 4 components so that $\lambda_2(\Omega)$  has multiplicity 3? 
  This is also a motivation for analyzing the cases $N=1, 2$.
 The natural conjecture (see Remark \ref{Remark4.3} for further discussion)  would be that
for simply connected domains $\Omega$, $\lambda_2(\Omega)$ has at most multiplicity~$2$.
\end{remark} 

\begin{remark}
 For a specific choice of the pair $(R_1,R_2)$ which will be introduced in Subsection  \ref{ss8.1}, the numerics (see Figure $2$) illustrates 
 the statement of  Theorem \ref{Theorem2.1} when $N=3$ and $N=4$. Although the precision is not very good for $\epsilon$ close to $0$ and $\frac{\pi}{3}$ (see Section \ref{snum}), we can  predict  as $N=3$  a second eigenvalue  of multiplicity $3$ for $\epsilon \sim 0.29 $.  A second crossing appears for $\epsilon \sim 0.96$ 
  but corresponds to a third eigenvalue of multiplicity $3$. The eigenvalues correspond (with the notation of Section \ref{s3}) to $\ell=0$ and to $\ell =1$, the eigenvalues 
  for $\ell=1$ having multiplicity $2$. When $N=4$, we also see a first crossing for $\epsilon \sim 0.54$ where the multiplicity becomes $3$, 
 as the theory will show. Two other crossings occur for $\epsilon=0.95$ and  $\epsilon=1.565$. The eigenvalues correspond (with the notation of Section \ref{s3}) to $\ell=0$, $\ell =1$ and $\ell=1$, the eigenvalues 
  for $\ell=1$ having multiplicity $2$.The eigenvalues for $\ell=0$ and $2$ are simple for $\epsilon\in (0,\frac \pi 2)$.
        
   \begin{figure}[!ht]\label{fig0}
 \begin{center}
\includegraphics[width=11.5cm]{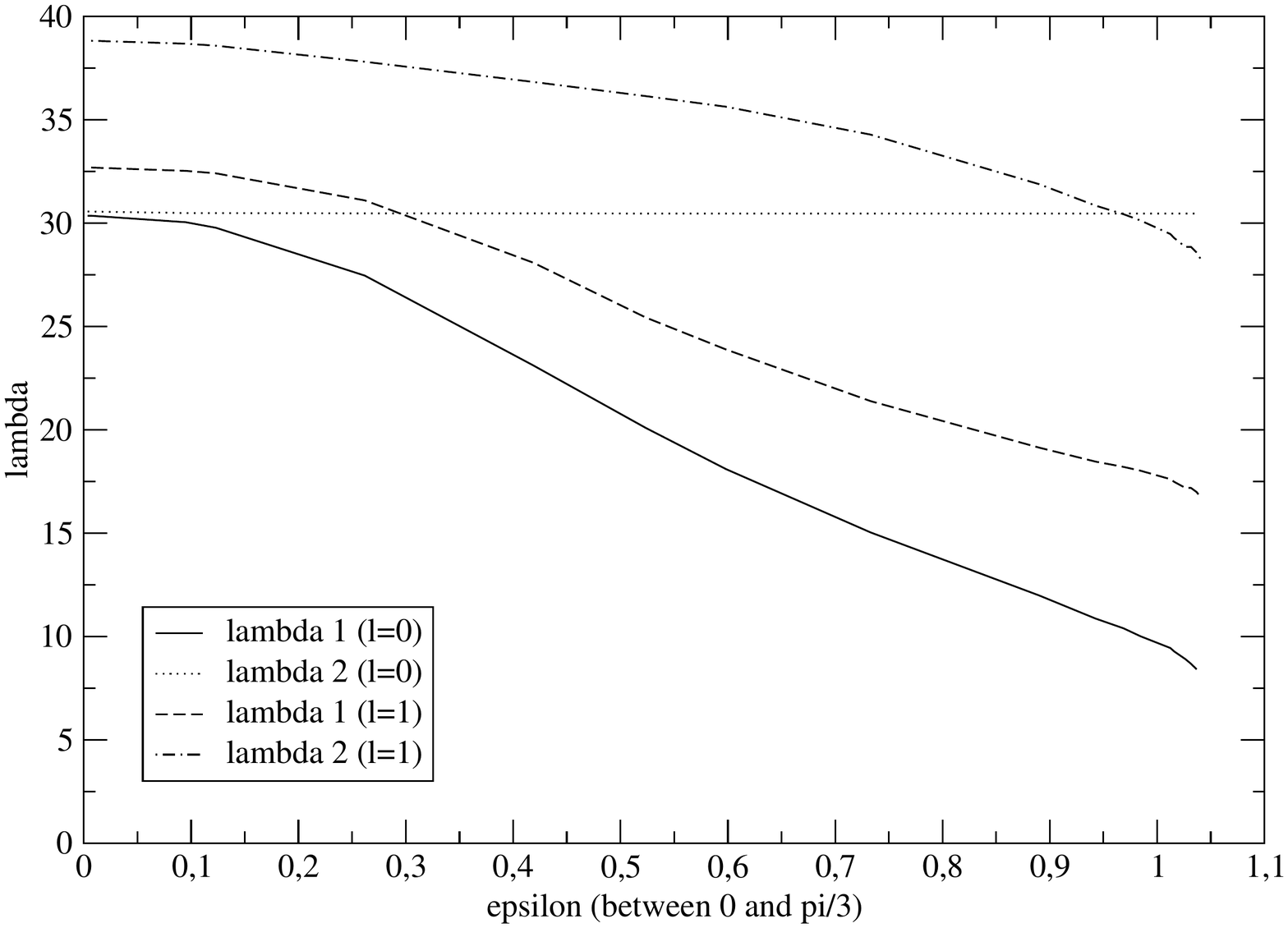}
 \caption{N=3. Six lowest eigenvalues of the Laplacian in $\mathfrak D(N,\epsilon))$  in function of $\epsilon \in (0,\frac \pi 3)$, with  $R_1= 0.4356$, $R_2=1$.} 
  \end{center}
 \end{figure}

    \begin{figure}[!ht]\label{figN=4}
 \begin{center}
\includegraphics[width=11.5cm]{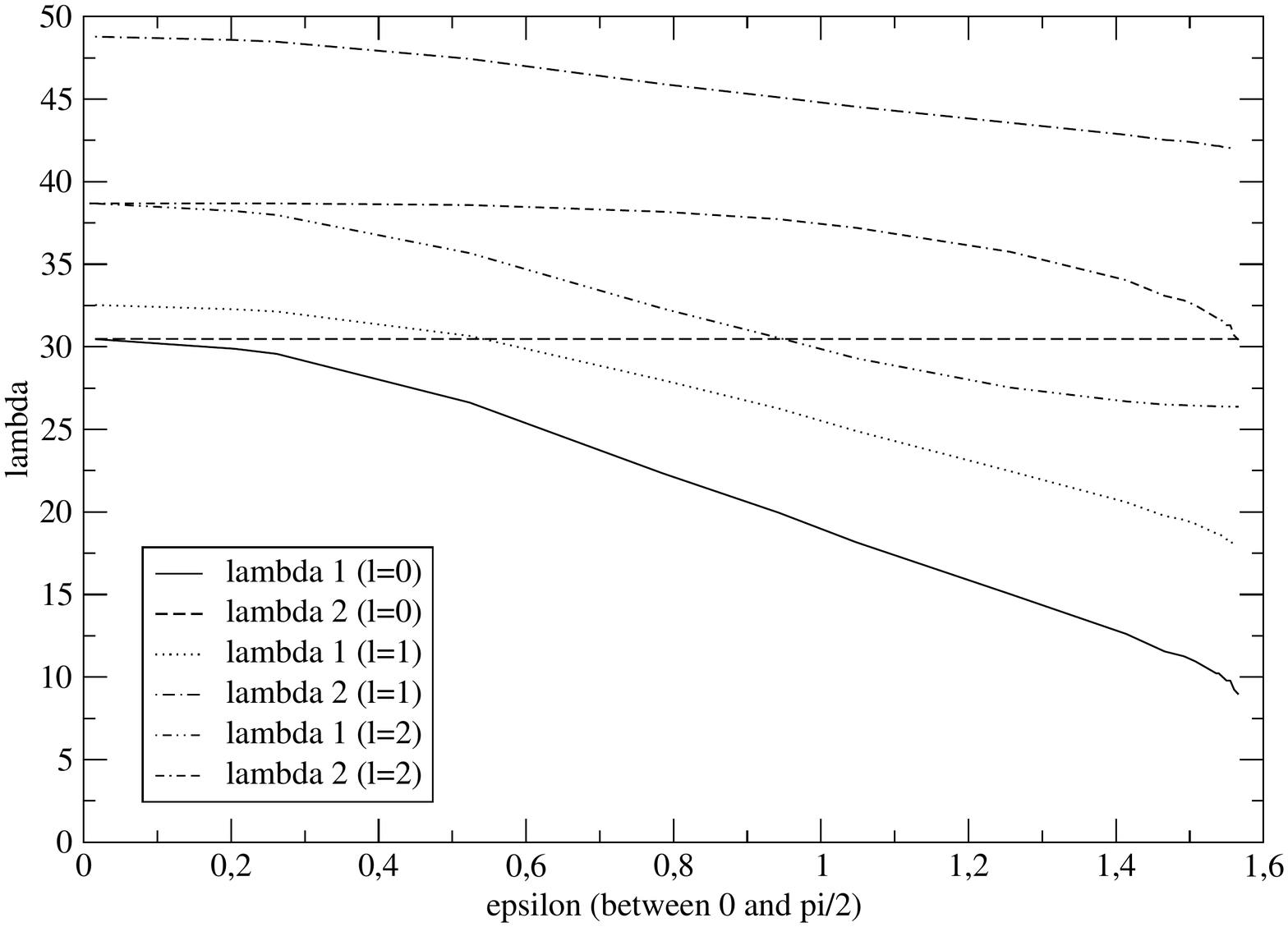}
 \caption{N=4. Eight lowest eigenvalues of the Laplacian in $\mathfrak D(N,\epsilon))$  in function of $\epsilon \in (0,\frac \pi 2)$, with  $R_1= 0.4356$, $R_2=1$.} 
  \end{center}
 \end{figure}

 \end{remark}

 We now  explain what were the difficulties arising in the sketch of the proof given in \cite{HOHON2}. \\
 
  The authors introduce a notion of symmetry or antisymmetry with respect to the inversion  $x\mapsto -x$ but  this does not work for odd $N$ since
$\mathfrak D(N,\ep)$ has no center of inversion. So the proof can only work for N even.\\

Considering $N$ even ($N\geq 4$), the idea behind the proof in \cite{HOHON2} is that there is a crossing for increasing $\epsilon$  between an eigenvalue associated with an antisymmetric eigenspace of multiplicity $2$ and 
an  eigenvalue associated with a symmetric eigenspace.  With the considered 
  antisymmetry proposed by the authors, it seems wrong that the multiplicity $2$ results simply  from the information that the eigenvalue corresponds to a non trivial antisymmetric eigenspace.  We will give a theoretical analysis in Section \ref{s7} 
   completed by a  numerical  study  in Section~\ref{snum} giving evidence that this guess is at least wrong in the simpler case $N=2$ which is not considered in \cite{HOHON2}. Hence  one has also to change the argument for even $N$.

 \section{Symmetry spaces}\label{s3}
  Before proving Theorem \ref{Theorem2.1},  we recall  some basic representation theory. 
We consider a Hamiltonian which is the Dirichlet realization of the Laplacian in an open set $\Omega$ which is invariant by the action of the group $G_N$ generated by the rotation $g$ by $\frac {2\pi}{N}$. The Hilbert space is $\mathcal H: = L^2(\Omega,\mathbb R)$ but it is also convenient to work in $\mathcal H_{\mathbb C}:= L^2(\Omega,\mathbb C)$. In this case,  it is natural to  analyze the eigenspaces attached to the irreducible representations of the group  $G_N$.  This is standard, see for example \cite{HHOHON} and references therein, but  note that these authors work with a larger group of symmetry, i.e. the dihedral group $\mathbb D_{2N}$. Here we prefer to start with the smaller group $G_N$  and it is important to note that  we do not assume in our work  that $\Omega$ is 
 homeomorphic to a disk or to an annulus. The theory of this section will in particular apply for the family of open sets $\Omega=  \mathfrak D(N,\epsilon) $ (which  satisfy the $\mathbb D_{2N}$-symmetry).  Hence in this case, 
  Theorems 1.2 and 1.3 of \cite{HHOHON} do not fully apply.\\
 The theory is simpler for  complex Hilbert spaces i.e. $\mathcal H_\mathbb C:=L^2(\Omega,\mathbb C)$, but the multiplicity property appears when considering operators on real Hilbert spaces, i.e $\mathcal H: =L^2(\Omega,\mathbb R)$. 
  If we work in $\mathcal H_\mathbb C$, we introduce for $\ell = 0, \cdots, N-1$,
 \begin{equation}\label {Bl}
\mathcal  B_\ell =\{w\in \mathcal H_{\mathbb C} \;|\;  g w= e^{ 2\pi i\ell/N}w\}\,.
\end{equation}
For $\ell=0$, this corresponds to the invariant situation. Hence in the model above (where $\Omega = B_{R_2}$)  $u_0$ and $u_6$  belong to $\mathcal B_0$. 
We also observe that the complex conjugation sends $\mathcal B_\ell$  onto $\mathcal B_{N-\ell}$. Hence, except  in the cases  $\ell=0$ and $\ell=\frac {N}{2}$ the corresponding eigenspace are of even dimension.\\ The second case appears only if $N$ is even.

For $2\ell \neq N$, one can alternately come back to real spaces by introducing  for $0< \ell < \frac  N2$ ($\ell \in \mathbb N$)
\begin{equation}
\mathcal C_\ell =\mathcal  B_\ell \oplus  \mathcal B_{N-\ell}\,
\end{equation}
 and observing that $\mathcal C_\ell$  can be recognized as the complexification
of the real space $\mathcal  A_\ell$
\begin{equation}\label{Al}
\mathcal  A_\ell=\{u\in \mathcal H\;|\; u-2\cos(2\ell\pi/N) g u+g^2u=0\}
\end{equation}
such that 
\begin{equation}
\mathcal C_\ell = \mathcal  A_\ell \otimes \mathbb C\;
\end{equation}
 where \eqref{Al} follows from an easy  computation based on \eqref{Bl}.\\
 For $\ell=0$ and $\ell=\frac  N 2$ (if $N$ is even), we define $\mathcal  A_\ell$ by
 \begin{equation}
\mathcal  B_\ell =  \mathcal  A_\ell \otimes \mathbb C\,.
 \end{equation}
 
 Under the invariance condition on the domain, the Dirichlet  Laplacian commutes with the natural action of $g$ in $L^2$. Hence we get for $0 \leq \ell \leq N/2$ a family of well defined selfadjoint operators $H^{(\ell)}$ obtained by restriction
  of $H$ to $\mathcal A_\ell$ (with domain $D(H)\cap \mathcal A_\ell$). Note that except for $\ell=0$ and $\ell = \frac {N}{2}$ all the eigenspaces of $H^{(\ell)}$ have even multiplicity.
  
  The other point is that Stollmann's theory \cite{Sto} works for the spectrum of $H^{(\ell)}(\epsilon,N)$ associated with the Dirichlet realization $H(\epsilon,N)$ of the Laplacian in $\mathfrak D_{N,\epsilon}$. Hence we have continuity and monotonicity with respect to $\epsilon$ of the eigenvalues. Note also that
  $$
  \sigma (H(\epsilon,N))=\cup_{0 \leq \ell\leq \frac  N2} \sigma (H^{(\ell)}(\epsilon,N))\,.
  $$
  
  \begin{remark}\label{rem3.1}
  When $N$ is even, a particular role is played by $g^\frac  N2$ which corresponds to the inversion considered in \cite{HOHON2}. One can indeed decompose the Hilbert space $\mathcal H$ (or $\mathcal H_\mathbb C$) using the symmetry with respect to $g^\frac  N2$ and get the decomposition
  \begin{equation}
  \mathcal H =\mathcal H^{even} \oplus \mathcal H^{odd}\,,
  \end{equation}
  and 
  \begin{equation}
  H(\epsilon,N)= H^{even}(\epsilon,N) \oplus H^{odd}(\epsilon,N)\,.
  \end{equation}
  One can compare this decomposition with the previous one. We observe that $\mathcal A_\ell$ belongs to $\mathcal H^{even} $ if $\ell$ is even and to  $\mathcal H^{odd} $
   if $\ell$ is odd.
  
  \end{remark}
  
    \section{Upper bound: the regularity assumptions in Cheng's statement revisited}
    In \cite{Ch}, S.Y. Cheng proved that the multiplicity of the second eigenvalue is at most~$3$. 
  Cheng's proof is actually using a regularity assumption which is not satisfied by $ D(N,\epsilon) $. This domain has indeed cracks and we need a  description  of the nodal line structure near corners or  cracks. But we will explain how to complete the proof in this case. We recall that for an eigenfunction $u$ the nodal set  $N(u)$ of $u$ is defined by
  $$
  \mathcal N (u):= \overline{\{x\in \Omega\,,|\, u(x)=0\}}\,.
  $$
  For other reasons (this was used in the context of spectral minimal partitions) this analysis was needed and treated  in the paper of Helffer, Hoffmann-Ostenhof, and Terracini \cite{HHOT} (Theorem~2.6).  With this complementary analysis 
   near the cracks, we can follow the main steps of the proof given in the first part of  \cite{HOMN} (Theorem B). This proof  includes an extended version of Euler's Polyhedral formula 
   (Proposition 2.8 in \cite{HOMN} with a stronger regularity assumption). 
  \begin{proposition} Let $\Omega$ be a $C^{1,+}$-domain\footnote{$C^{1,+}$ means $C^{1,\epsilon}$ for some $\epsilon >0$.} with possibly corners of opening\footnote{$\alpha=2$ corresponds to the crack case.} $\alpha \pi$ for $0 < \alpha\leq 2$.
   If $u$ is an eigenfunction of the Dirichlet Laplacian in $\Omega$, $\mathcal N$ denotes the nodal set of $u$ and $ \mu(\mathcal N) $ denotes the cardinality of the components of $\Omega \setminus \mathcal N$,   i.e. the number of nodal domains, then
  \begin{equation}\label{Euler}
  \mu(\mathcal N) \geq \sum_{x\in \mathcal N\cap \Omega} (\nu(x) -1) +2\,,
  \end{equation}
  where $\nu(x)$ is the multiplicity of the critical point $x\in \mathcal N$ (i.e. the number of lines crossing at $x$).
  \end{proposition}
   For a second eigenfunction $\mu(\mathcal N)=2$, and the upper bound of the multiplicity by $3$ comes by  contradiction. Assuming that the multiplicity of the second eigenvalue   is $\geq 4$, one can, for any $x\in \Omega$, construct some $u$ in the second eigenspace such that $\nu(x)\geq 2$. This gives the contradiction with \eqref{Euler}.  Hence we have
   \begin{proposition}
    Let $\Omega$ be a $C^{1,+}$-domain with possibly corners of opening $\alpha \pi$ for $0 < \alpha\leq 2$. Then the multiplicity of the second eigenvalue of the Dirichlet Laplacian in $\Omega$ is not larger than $3$.
    \end{proposition}
  
    \begin{remark}\label{Remark4.3}
    An upper bound of the multiplicity by $2$ is obtained by C.S. Lin when $\Omega$ is convex (see \cite{Lin}). As observed at the end of Section 2 in \cite{ZL}, Lin's theorem can be extended to the  case of a simply connected domain for which the nodal line conjecture holds. If the multiplicity of the second  eigenvalue is larger than 2,   one can indeed find in the associated spectral space an eigenfunction whose nodal set contains a point in the boundary where two half lines hit the boundary.  This will contradict either the nodal line conjecture or Courant's theorem. See also \cite{HOHON2}  for some sufficient conditions on domains for the nodal
line conjecture to  hold.
      \end{remark}

 \section{Proof of Theorem \ref{Theorem2.1}}
 
 We first observe that for the disk of radius $R$ we have
 \begin{equation}\label{specdisc}
 \lambda_1(B_{R})<  \lambda_2(B_{R})= \lambda_3(B_{R}) <  \lambda_4(B_{R})= \lambda_5(B_{R}) < \lambda_6(B_{R})\,.
 \end{equation}
The eigenfunctions  $u_1$ and $u_6$ are radial. We will use this property with $R=R_2$.
 \begin{proposition}
 For $N\geq 3$, there exists $\epsilon \in (0, \frac {\pi}{N})$ such that $\lambda_2(H(\epsilon,N))$ belongs to $\sigma (H^{(\ell)}(\epsilon,N))$ for some $0 < \ell < \frac  N2$ AND to $\sigma (H^{(\ell)}(\epsilon,N))$ 
  for $\ell=0$ or (in the case $N$ even) $\frac  N2$. 
  In  particular, the multiplicity of $\lambda_2$ for this value of $\epsilon$ is exactly $3$.
  \end{proposition}
  \begin{proof}~\\
  Note that the condition $N\geq 3$ implies the existence of at least one $\ell\in (0,\frac  N2)$.\\
  We now proceed by contradiction. Suppose the contrary. By continuity of the second eigenvalue, we should have
  \begin{itemize}
  \item  either $\lambda_2(H(\epsilon, N))$ belongs to $ \cup_{0 < \ell < \frac  N2} \sigma (H^{(\ell)} (\epsilon, N)))$ and  not to \break 
   $\sigma (H^{(0)} (\epsilon, N))\cup \sigma (H^{(N/2)} (\epsilon, N))$ for any $\epsilon$, 
  \item  or
   $\lambda_2(H(\epsilon, N))$ belongs to   $\sigma (H^{(0)} (\epsilon, N))\cup \sigma (H^{(N/2)} (\epsilon, N))$   and not to  \break $ \cup_{0 < \ell < \frac  N2} \sigma (H^{(\ell)} (\epsilon, N)))$ for any $\epsilon$.
   \end{itemize}
   But, as we shall see below,  the analysis for $\epsilon>0$ small enough shows that we should be in the first case and the analysis for $\epsilon$ close to $\frac  \pi N$  that  we should be in the second case. Hence a contradiction.\\
   
   The analysis for $\epsilon >0$ very small is by perturbation a consequence of the analysis of $\epsilon=0$. Here we see from \eqref{specannulus} that $\lambda_2 (D_{R_1,R_2})$ is simple and belongs to $\sigma (H^{(0)} (0,N))$.\\
   \begin{remark}
   If we only have \eqref{ineqb}, we observe that the two first eigenvalues belong to $\sigma (H^{(0)} (0,N))$ and the argument is unchanged.
   \end{remark}
   The analysis for $\epsilon$ close to $\frac  \pi N$ is by perturbation a consequence of the analysis of $\epsilon =\frac  \pi N$. More details (which are not necessary for the argument) will be given in Section \ref{s7}. Here we see from \eqref{specdisc} 
   that $\lambda_2 (B_{R_2})$  has multiplicity two corresponding to $\sigma (H^{(1)} (\frac  \pi N,N))$.\\
   
   So we have proven that for this value of $\epsilon$ the multiplicity is at least three, hence equals three by the extension of Cheng's statement \cite{Ch}  proven in the previous section.
  \end{proof}
 
 \paragraph{Comparison with the former proof  proposed in \cite{HOHON2}}~\\
 When $N/2$ is even, we deduce from Remark \ref{rem3.1} that
 $$
 \sigma (H^{(odd)}) \subset \cup_{0 < \ell < \frac  N2} \sigma (H^{(\ell)} (\epsilon, N))) \mbox{ and } \sigma (H^{(1)} (\epsilon, N)) \subset  \sigma (H^{(odd)}) \,,
 $$
 with equality for $N=4$.\\
 From these two properties which imply that the eigenvalues in  $ \sigma (H^{(odd)}) $ have even multiplicity we can rewrite the previous proof in the way presented in \cite{HOHON2}:
 \begin{proposition}
 For $N\geq 4$ and $N/2$ even, there exists $\epsilon \in (0, \frac {\pi}{N})$ such that $\lambda_2(H(\epsilon,N))$ belongs to $\sigma (H^{even}(\epsilon,N))$  AND to $\sigma (H^{odd}(\epsilon,N))$.
  In  addition, the multiplicity of $\lambda_2$ for this value of $\epsilon$ is exactly $3$.
  \end{proposition}

 For $N/2$ odd  integer some extra argument  is necessary to exclude that an eigenvalue in $\sigma (H^{(N/2)} (\epsilon, N))$ (which belongs to  $\sigma (H^{(odd)}) $)
  becomes a second eigenvalue. More precisely we should prove that $\lambda_1(H^{(N/2)} (\epsilon, N)) > \lambda_1(H^{(1)} (\epsilon, N))$ for any $\epsilon >0$. Here we have to use the additional dihedral invariance and use the arguments in \cite{HHOHON}.  The inequality follows from comparing the nodal sets of corresponding eigenfunctions (see (3.3) and (3.4) in \cite{HHOHON} 
  after having verified that the proof does not use the assumption that $\Omega$ is homeomorphic to a disk or an annulus).
   Hence, we have completed the proof sketched in \cite{HOHON2} but the new proof looks more natural.

 \section{Further discussion for the case $N=2$}\label{s6}
 In the previous sections, we have excluded the case $N=2$ because we were unable to prove that the eigenspaces of $\mathcal H^{(N/2)}$  have even dimension and there
  were no more spaces $\mathcal A_\ell$ with $0 < \ell < \frac  N2$ to play with. We now assume $N=2$ and consider $\mathfrak D (2,\epsilon)$. Note that this time it will be quite important to have not only the dihedral symmetry but also the property that the cracks are on a circle.\\
  As in \cite{HHOHON} (see (1.16) and (1.17) there), we will use the decomposition of $L^2$:
  $$
 \mathcal H:= L^2 (\mathfrak D(2,\epsilon))= \mathcal A_0 \oplus \mathcal A_1^{a} \oplus \mathcal A_1^{s}\,.
  $$
  
  Here
  $$
  \begin{array}{ll}
  \mathcal A_1^{s}&= \{ u \in \mathcal H\,,\, g u = -u\,,\, Tu =u\}\,,\\
   \mathcal A_1^{a}&= \{ u \in \mathcal H\,,\, g u = -u\,,\, Tu =- u\} \,,
   \end{array}
   $$
   where $T u (x_1,x_2)= u (x_1,-x_2)$.\\ We also observe that $g$ is for $N=2$ the inversion.\\
   We similarly define the operators $H^{(1,a)}(\epsilon,2)$ and $H^{(1,s)}(\epsilon,2)$.
   The question is then to compare the spectra of these two operators and more specifically the first eigenvalue.\\
   If we observe what is imposed by the symmetry or the antisymmetry with respect to $\{x_1=0\}$ or $\{x_2=0\}$ we can replace $\mathfrak D(2,\epsilon)$ by
  $$ \widehat {\mathfrak D}(2,\epsilon):= \mathfrak D(2,\epsilon)\cap \{x_1>0, x_2>0\}\,.$$ The problem corresponding to $H^{(1,s)}(\epsilon,2)$ is the problem where we assume on $\partial  \widehat {\mathfrak D}(2,\epsilon)\cap\{x_2=0\}$ the Neumann condition and on $\partial  \widehat {\mathfrak D}(2,\epsilon)\cap\{x_1=0\}$ the Dirichlet condition, keeping the Dirichlet condition on the other parts of the boundary.
  
  The problem corresponding to $H^{(1,a)}(\epsilon,2)$ is the problem where we assume on $\partial  \widehat {\mathfrak D}(2,\epsilon)\cap\{x_2=0\}$ the Dirichlet condition and on $\partial  \widehat {\mathfrak D}(2,\epsilon)\cap\{x_1=0\}$ the Neumann condition, keeping the Dirichlet condition on the other parts of the boundary.\\
  For $\epsilon =0$ and $\epsilon =\frac  \pi 2$ the two operators are isospectral. Hence the question is:\\
    Are the ground state energies of the two problems the same or are they different 
   for $\epsilon \neq 0, \frac  \pi 2$?
   \\
  We will show in Section \ref{s7} that for a given pair $(R_1,,R_2)$ with $R_1<R_2$ this can only be true,   in any closed subinterval of $(0,\frac \pi 2]$, for a finite  number  of different values of  $\epsilon$'s. Moreover,  for a specific natural pair $(R_1,R_2)$ we can give in Section~\ref{snum} the following numerically assisted answer:\\
  {\it  The ground state energies of $H^{(1,a)}(\epsilon,2)$ and $H^{(1,s)}(\epsilon,2)$ are equal if and only if $\epsilon=0$ or $\frac  \pi 2$.}

  \section{ Theoretical asymptotics in domains with cracks}\label{s7}
In this section, we analyze theoretically the behavior of the eigenvalue as $\epsilon$ tends to $\frac \pi 2$. This improves the general results based on \cite{Sto} and 
explains why we have  to modify the sketch of \cite{HOHON2} for the proof of Theorem \ref{Theorem2.1}.

 \subsection{Preliminaries}\label{ss7.1}

We now fix $N=2$ and consider $0<R_1<R_2$. Motivated by the previous question, we analyze the different spectral problems according to the symmetries. This leads us to consider on the quarter of a disk ($0< \theta< \frac \pi 2$)
 four different models. On the exterior circle and on the cracks, we always assume the Dirichlet condition and then, according to the boundary conditions retained for $\theta=0$ and $\theta=\pi/2$,  we consider four test cases~:
\begin{itemize}
\item Case NND (homogeneous Neumann boundary conditions for $\theta=0$ and $\theta=\pi/2$). 
\item Case DDD (homogeneous Dirichlet boundary conditions for $\theta=0$ and $\theta=\pi/2$). 
\item Case NDD (homogeneous Neumann boundary conditions for $\theta=0$ and homogeneous Dirichlet boundary conditions for $\theta=\pi/2$). 
\item Case DND (homogeneous Dirichlet boundary conditions for $\theta=0$ and homogeneous Neumann boundary conditions for $\theta=\pi/2$). 
\end{itemize}
This is immediately related to the problem on  the  cracked disk by using the symmetries with respect to the two axes. The symmetry properties lead either to Dirichlet or Neumann.

\subsection{ The cases NND and DND}\label{ss8.1}
We use the notation
$$
\begin{array}{ll}
B_{R_2}^+&:=B_{R_2}\cap\{x_2>0\}\,;\\
x_\pm&:=(0,\pm R_1)\,;\\
\delta&:=\frac\pi2-\epsilon\,;\\
K_\delta&:=\{x\in\mathbb R^2\,;\,r=R_1, \theta\in[-\pi/2-\delta,-\pi/2+\delta]\cap[\pi/2-\delta,\pi/2+\delta]\}\,;\\
K_\delta^+&:=K_\delta\cap\{x_2>0\}\,;\\
K_\delta^-&:=K_\delta\cap\{x_2<0\}\,.
\end{array}
$$
By the symmetry arguments of Section \ref{s6},
\begin{align*}
\lambda_1^{NND}( \widehat{\mathfrak D}(2,\epsilon))&=\lambda_1(B_{R_2}\setminus K_\delta);\\
\lambda_1^{DND}( \widehat{\mathfrak D}(2,\epsilon))&=\lambda_1(B_{R_2}^+\setminus K_\delta^+).
\end{align*}
The family of compact sets $(K_\delta)_{\delta>0}$ concentrates to the set $\{x_+,x_-\}$, in the sense that $K_\delta$ is contained in any open neighborhood of $\{x_+,x_-\}$ for $\delta$ small enough. Reference \cite{AFHL18} provides  two-term asymptotic expansions in this situation.

A direct application of Theorem 1.7 in \cite{AFHL18} gives 
\begin{equation*}
	\lambda_1(B_{R_2}^+\setminus K_\delta^+)=\lambda_1(B_{R_2}^+)+u(x_+)^2\frac{2\pi}{\left|\log(\mbox{diam}(K_\delta^+)\right|}+o\left(\frac1{\left|\log(\mbox{diam}(K_\delta^+)\right|}\right),
\end{equation*}
where $\mbox{diam}(K_\delta^+)$ is the diameter of $K_\delta^+$ and $u$ an eigenfunction associated with $\lambda_1(B_{R_2}^+)$, normalized in $L^2(B_{R_2}^+)$.  Using $ \mbox{diam}(K_\delta^+)=2R_1\sin(\delta)$ and the normalized eigenfunction given by Proposition 1.2.14 in \cite{Henrot06} we find, after simplification
\begin{multline}
\label{eqDND}
\lambda_1^{DND}(\widehat{\mathfrak D}(2,\epsilon))=j_{1,1}^2+\\
+ \frac8{R_2^2}\left(\frac{J_1(j_{1,1}R_1/R_2)}{J_1'(j_{1,1})}\right)^2\frac1{\left|\log(\pi/2-\epsilon)\right|}+o\left(\frac1{\left|\log(\pi/2-\epsilon)\right|}\right),
\end{multline}
 where $j_{\ell,k}$ is the $k$-th zero of the Bessel function $J_\ell$ corresponding to
 the integer $\ell\in \mathbb N$ (see Subsection \ref{ss7.2} for more details and numerical values).\\

We obtain a similar expansion for the other eigenvalue, starting from Theorem 1.4 in \cite{AFHL18}, which gives us
\begin{equation*}
	\lambda_1(B_{R_2}\setminus K_\delta)=\lambda_1(B_{R_2})+{\rm Cap}_{B_{R_2}}(K_\delta,u)+o\left({\rm Cap}_{B_{R_2}}(K_\delta,u)\right).
\end{equation*}
In this formula, $u$ is an eigenfunction associated with $\lambda_1(B_{R_2})$ and normalized in $L^2(B_{R_2})$, and ${\rm Cap}_{B_{R_2}}(K_\delta,u) $ is defined by Equation (6) in \cite{AFHL18}. Since $u$ is radially symmetric, $u(x_+)=u(x_-)$. We then observe that the proof of Proposition 1.5 in \cite{AFHL18} can be adapted to give 
\begin{equation*}
	{\rm Cap}_{B_{R_2}}(K_\delta,u)=u(x_\pm)^2{\rm Cap}_{B_{R_2}}(K_\delta) +o\left({\rm Cap}_{B_{R_2}}(K_\delta)\right),
\end{equation*}
where ${\rm Cap}_{B_{R_2}}(K_\delta)$ is the classical (condenser) capacity of $K_\delta$ relative to $B_{R_2}$. Since $K_\delta=K_\delta^+\cup K_\delta^-$, and since $K_\delta^+$ and $K_\delta^-$ concentrate to $x_+$ and $x_-$ respectively, we have
\begin{equation*}
	{\rm Cap}_{B_{R_2}}(K_\delta)\sim{\rm Cap}_{B_{R_2}}(K_\delta^+)+{\rm Cap}_{B_{R_2}}(K_\delta^-)
\end{equation*}
as $\delta\to 0$. This last fact seems to be well known  (see \cite{Fl}, page 178), but we give a proof in Appendix \ref{appCapAdd} for completeness. Finally, Proposition 1.6 in \cite{AFHL18} gives an asymptotic expansion for ${\rm Cap}_{B_{R_2}}(K_\delta^\pm)$. Gathering these estimates, we find

\begin{multline}
\label{eqNND}
\lambda_1^{NND}( \widehat{\mathfrak D}(2,\epsilon))=j_{0,1}^2+\\
+ \frac4{R_2^2}\left(\frac{J_0(j_{0,1}R_1/R_2)}{J_0'(j_{0,1})}\right)^2\frac1{\left|\log(\pi/2-\epsilon)\right|}+o\left(\frac1{\left|\log(\pi/2-\epsilon)\right|}\right).
\end{multline}

 \subsection{Analysis of the cases NDD and DDD}
 In these cases, the results in \cite{AFHL18} give  an estimate of the eigenvalue variation but no explicit first term for the expansion. However, they strongly suggest the form of this term, which we present as a conjecture in each case. By the symmetry arguments of Section \ref{s6}, we have, for $\epsilon$ close to $\pi/2$,
\begin{equation*}
	\lambda_1^{DDD}(\widehat{\mathfrak D}(2,\epsilon))=\lambda_2(B_{R_2}^+\setminus K_\delta^+).
\end{equation*}
We further note that $\lambda_2(B_{R_2}^+)$ ($=j_{2,1}^2$) is simple and that an associated eigenfunction $u$, normalized in $L^2(B_{R_2}^+)$, is given by 
\begin{equation*}
	u(r\cos \theta,r\sin \theta)=\frac2{\sqrt\pi}\frac1{R_2|J'_2(j_{2,1})|}J_2\left(\frac{j_{2,1}r}{R_2}\right)\sin(2\theta).
\end{equation*}
In particular, it follows that 
\begin{equation*}
	\partial_{x_1}u(x_+)= \frac1{R_1}\frac{4}{\sqrt\pi}\frac1{R_2|J'_2(j_{2,1})|}J_2\left(\frac{j_{2,1}R_1}{R_2}\right)\,,
\end{equation*}
so that,
\begin{equation*}
	u(x_++\rho(\cos t,\sin t))= \frac4{\sqrt\pi}\frac1{R_1 R_2|J'_2(j_{2,1})|}J_2\left(\frac{j_{2,1}R_1}{R_2}\right)\rho\cos t+\mathcal O \left(\rho^2\right).
\end{equation*}
Theorem 1.4 in \cite{AFHL18} gives us 
\begin{equation}
\label{eqDDDCap}
	\lambda_1(B_{R_2}^+\setminus K_\delta^+)=\lambda_1(B_{R_2}^+)+{\rm Cap}_{B_{R_2}^+}(K_\delta^+,u)+o\left({\rm Cap}_{B_{R_2}^+}(K_\delta^+,u)\right).
\end{equation}

\begin{proposition}\label{propDDD}
There exists $C>0$ such that
\begin{equation}
\label{eqDDDw}
j_{2,1}^2 \leq \lambda_1^{DDD}(\widehat{ \mathfrak D }(2,\epsilon))\leq j_{2,1}^2+ C \left(\left(\frac\pi2-\epsilon \right)^2\right).
\end{equation}
\end{proposition}

\begin{proof} By monotonicity of the Dirichlet eigenvalues with respect to the domain, we immediately have
$$
j_{2,1}^2= \lambda_1^{DDD}(\widehat{ \mathfrak D }(2,\pi/2))\leq \lambda_1^{DDD}(\widehat{ \mathfrak D }(2,\epsilon)). 
$$
On the other hand, since $x_+$ is a (regular) point in the nodal set of $u$, we have ${\rm Cap}_{B_{R_2}^+}(K_\delta^+,u)=O\left(\delta^2\right)$ as $\delta\to0$, according to Lemma 2.2 in \cite{AFHL18}. The upper bound follows from Equation \eqref{eqDDDCap}.
\end{proof}

\begin{conjecture}\label{conjCap1}
As $\delta\to 0^+$, 
\begin{equation*}
	{\rm Cap}_{B_{R_2}^+}(K_\delta^+,u)\sim{\rm Cap}_{B_{R_2}^+}(s_\delta,u),
\end{equation*}
where $s_\delta$ is the segment $ [-\delta R_1,\delta R_1]\times\{R_1\}$.
\end{conjecture}

The first term of the asymptotics of ${\rm Cap}_{B_{R_2}^+}(s_\delta,u)$ is given by Theorem 1.13 in \cite{AFHL18},
 so that Conjecture \ref{conjCap1} would imply that as $\epsilon \to\pi/2^-$,
\begin{multline}
\label{eqDDD}
\lambda_1^{DDD}(\widehat{ \mathfrak D }(2,\epsilon))=j_{2,1}^2+\\
+ \frac{16}{R_2^2}\left(\frac{J_2(j_{2,1}R_1/R_2)}{J_2'(j_{2,1})}\right)^2\left(\frac\pi2-\epsilon \right)^2+o\left(\left(\frac\pi2-\epsilon \right)^2\right).
\end{multline}
Using again Section \ref{s6}, we have
\begin{equation*}
	\lambda_1^{NDD}(\widehat{ \mathfrak D }(2,\epsilon))=\lambda_2(B_{R_2}\setminus K_\delta)=\mu_2(B_{R_2}^+\setminus K_\delta^+),
\end{equation*}
where $(\mu_j)_{j\ge1}$ denotes the eigenvalues of the Laplacian in $B_{R_2}^+$, with a Dirichlet condition on the semi-circle $\{|x|=R_2\}\cap{x_2>0}$ and a Neumann condition on the diameter $[-R_2,R_2]\times\{0\}$. We note that $\mu_2(B_{R_2}^+)$ is simple. The eigenvalue $\lambda_2(B_{R_2})$ is double when the Laplacian is understood as acting on $L^2(B_{R_2})$, but simple if we restrict the Laplacian to   $\mathcal A_1^{s}$  (as defined in Section \ref{s6}).
We denote by $(\lambda_j^s)_{j\ge1}$ the spectrum of  $H^{(1,s)}(\epsilon,2)$ (by a slight abuse of notation we do not specify the value of $\epsilon$). We remark that 
\[\mu_2(B_{R_2}^+)=\lambda_2(B_{R_2})=  \lambda_1^s (B_{R_2}).\]
An eigenfunction associated with $\lambda_1^s(B_{R_2})$ and normalized in $\mathcal A_1^s$ is given by 
\begin{equation*}
	u(r\cos \theta,r\sin \theta)=\sqrt{\frac2{\pi}}\frac1{R_2|J'_1(j_{1,1})|}J_1\left(\frac{j_{1,1}r}{R_2}\right)\cos \theta\,.
\end{equation*}
An eigenfunction associated with $\mu_2(B_{R_2}^+)$ and normalized in $L^2(B_{R_2}^+)$ is given by $\sqrt2 u_{|B_{R_2}^+}$.

\begin{proposition}\label{propNDD}
There exists $C>0$ such that
\begin{equation}
	\label{eqNDDw}
j_{1,1}^2 \leq \lambda_1^{NDD}(\widehat{ \mathfrak D }(2,\epsilon))\leq j_{1,1}^2+ C \left(\frac\pi2-\epsilon \right)^2 .
\end{equation}
\end{proposition}

\begin{proof} We use $\lambda_1^{NDD}(\widehat{ \mathfrak D }(2,\epsilon))=\lambda_2(B_{R_2}\setminus K_\delta)$. To find an upper bound, we would like to apply Theorem 1.4 in \cite{AFHL18} to get 
$$
	\lambda_2(B_{R_2}\setminus K_\delta)=\lambda_2(B_{R_2})+{\rm Cap}_{B_{R_2}}(K_\delta,u)+o\left({\rm Cap}_{B_{R_2}}(K_\delta,u)\right).
$$
Unfortunately, this cannot be done directly, since $\lambda_2(B_{R_2})$ is a double eigenvalue. But the result is easily obtained by repeating the proof in Appendix A of \cite{AFHL18} for the Laplacian acting on the 
 symmetric space $\mathcal A_1^s$. A straightforward adaptation of the proof of Lemma 2.2 in \cite{AFHL18}, taking into account the fact that $K_\delta$ concentrates to two points, gives ${\rm Cap}_{B_{R_2}}(K_\delta,u)=O\left(\delta^2\right)$ as $\delta\to0$.
\end{proof}

Let us assume again that Conjecture \ref{conjCap1} holds. Let us also assume that Theorem 1.4 in \cite{AFHL18} holds for the eigenvalue problems with mixed boundary  conditions which define $(\mu_j)_{j\ge1}$. We obtain, as $\delta\to0^+$,
\begin{equation*}
	\mu_2(B_{R_2}^+\setminus K_\delta^+)=\mu_2(B_{R_2}^+)+2\pi|\partial_{x_1}u(x_+)|^2R_1^2\delta^2+o\left(\delta^2\right).
\end{equation*}
Alternatively, we could work  in the 
symmetric space $\mathcal A_1^s$. Repeating the proof in Appendix A of \cite{AFHL18} in this space and assuming that the $u$-capacity defined in \cite{AFHL18} is asymptotically additive for small distant sets, we obtain
\begin{equation*}
	\lambda_1^s(B_{R_2}\setminus K_\delta)= \lambda_1^s(B_{R_2})+\pi|\partial_{x_1}u(x_+)|^2R_1^2\delta^2+\pi|\partial_{x_1}u(x_-)|^2R_1^2\delta^2+o\left(\delta^2\right).
\end{equation*}
Both methods would give
\begin{multline}
	\label{eqNDD}
\lambda_1^{NDD}(\widehat{ \mathfrak D }(2,\epsilon))=j_{1,1}^2+\\
+ \frac{4}{R_2^2} \left(\frac{J_1(j_{1,1}R_1/R_2)}{J_1'(j_{1,1})}\right)^2\left(\frac\pi2-\epsilon \right)^2+o\left(\left(\frac\pi2-\epsilon \right)^2\right).
\end{multline}
\subsection{Comparison}

As a consequence of Propositions  \ref{propDDD} and \ref{propNDD} and using also the analyticity with respect to $\epsilon$, we obtain
\begin{proposition}\label{prop7.4}
There exists $\epsilon_0 \in (0,\frac \pi 2)$ such that for $\epsilon \in [\epsilon_0,\frac \pi 2)$ we have
$$
\lambda_1^{NDD}(\widehat{ \mathfrak D }(2,\epsilon)) < 
\lambda_1^{DND}(\widehat{ \mathfrak D }(2,\epsilon)) \,.
$$
Moreover $\delta(\epsilon):=\lambda_1^{NDD}(\widehat{ \mathfrak D }(2,\epsilon)) -
\lambda_1^{DND}(\widehat{ \mathfrak D }(2,\epsilon))$ can at most  vanish  in $(0, \frac \pi 2)$    on a sequence of $\epsilon$'s with no accumulation point except possibly at $0$.
\end{proposition}
 A more accurate analysis as $\epsilon \rightarrow 0$ would be useful for excluding the possibility of a sequence of zeros of $\delta$ tending to $0$.
We will see in the next section that numerics strongly suggests that  $\delta(\epsilon)$  is negative in  $(0, \frac \pi 2)$.
The argument used in Proposition \ref{prop7.4} is general and not related to $N=2$.

 \section{Some illustrating numerics}\label{snum}
  \subsection{Preliminaries}\label{ss9.1}
In this section, we complete the theoretical study of the previous section by using numerics. 
 With the discussion around \eqref{specannulusa} in mind, a  particular choice for the pair $(R_1,R_2)$ is  to  start from $B_{R_2}$ with $R_2=1$, and then to take as $0 < R_1< 1$ the radius of the circle on which the second radial eigenfunction (which is associated with the sixth eigenvalue)  vanishes. 
 In this case, we have 
$ \lambda_1(B_{R_1}) = \lambda_1 ( M_{R_1,R_2} )$ and $\lambda_1(B_{R_1})=\lambda_6(B_{R_2})$ is an eigenvalue of the Dirichlet Laplacian in $ \mathfrak D(N,\epsilon) $ for any $\epsilon \in [0,\frac  \pi N]$. Its labelling as eigenvalue is $2$ for 
$0 < \epsilon \leq \epsilon_0$ and becomes $6$ for $\epsilon$ sufficienly close $\frac  \pi N$. In addition, there is a unique $\epsilon_0^*$ such that the labelling is $2$ for $0 < \epsilon \leq \epsilon_0^*$  and becomes $>2$ for $\epsilon >\epsilon_0^*$. This follows from the piecewise analyticity of the eigenvalues (Kato's theory)  and a more detailed analysis  as $\epsilon \rightarrow 0$ or $\epsilon \rightarrow \frac  \pi N$ (see \cite{DH} for the technical details,   \cite{CP,HJ} for related questions and our previous  section).
  The two next subsections recall what will be used for a theoretical verification of our numerical approach in the limits $\epsilon \rightarrow 0$ and $\epsilon \rightarrow +\frac \pi 2$. We have indeed in these cases enough theoretical
  information  for controlling the method.
  
 \subsection{Reminder: the case of the disk (Dirichlet)}\label{ss7.2}
  Let $j_{\ell,k}$ be the $k$-th zero of the Bessel function $J_\ell$ corresponding to
 the integer $\ell\in \mathbb N$. Here is a list of approximate values
 after the celebrated handbook of \cite{AS}, p. 409, we keep only the first values including at least all the eigenvalues which are less than approximately $13$.

\begin{equation}
\begin{array}{rcccccccccccc}
&\ell=&0 &1 &2 & 3 & 4 & 5 & 6&7&8\\
k=
 1 && 2.404& 3.831 & 5.135 &6.380&7.588& 8.771&9.936&11.086&12.225&\\
2& & 5. 520 & 7.015 & 8. 417 & 9.761 & 11. 064 & 12.338& 13.589& 14.821& 16.037&\\
3& &8,653&10.173&11.619 & 13.015& 14.372& 15.700& 17.003& 18.287& 19.554&\\
4& & 11.791&13.323& 14.796& 16.223& 17.616& 18.980& 20.320& 21.641& 22.942&
  \end{array}
\end{equation}

This leads to the following ordering of the zeros :
\begin{equation}
\begin{array}{l}
j_{0,1}< j_{1,1} <
 j_{2,1} < j_{0,2}< j_{3,1}<j_{1,2}<j_{4,1}<j_{2,2}<j_{0,3}<\dots\\
\qquad \dots  < j_{5,1} < j_{3,2} < j_{6,1} < j_{1,3} < j_{7,1} < j_{2,3}
 < j_{0,4} < j_{8,1}\;.
\end{array}
\end{equation}
The corresponding eigenvalues for the disk $B_1$ of radius $1$, with Dirichlet condition 
are given by $j_{\ell, k}^2$. The multiplicity is $1$ if $\ell=0$ (radial case)  and $2$ if $\ell\neq 0$.\\
Hence we get for the six first eigenvalues (ordered in increasing order):
\begin{equation}\label{eq:7.3}
\begin{array}{l}
\lambda_1^D(B_1)\sim 5.78  \\
\lambda_2^D(B_1)= \lambda_3^D(B_1)\sim  14.68 \\
\lambda_4^D(B_1)= \lambda_5^D(B_1) \sim 26.37 \\
\lambda_6^D(B_1)\sim  30.47\,.
\end{array}
\end{equation}
With our choice of the pair $(R_1,R_2)$ in the previous subsection,  we have 
\begin{equation}\label{defR1}
R_1 = j_{0,1}/j_{0,2} \sim 0.4356\,.
\end{equation}
\subsection{The case of the annulus with $R_1=0.4356$}\label{ss7.3}
We only keep the eigenvalues which are smaller than $150$. Note that the multiplicity is $2$ as soon as $\ell >0$. The precision is relatively good but of no use due to the fact that we also use some approximation of the right $R_1$ defined in \eqref{defR1}.

\begin{equation}
\begin{array}{rccccccccccccc}
&\ell=&0 &1 &2 & 3 & 4 & 5 & 6&7&8\\
k= 1 && 30.46&  32.53 &  38.68 &  48.78&   62.61&  79.91&  100.39&    123.79&  149.90&\dots

 \\
2 &&  123.38& 125.60&  132.29 &    143.45& \dots 
\\
3 && \dots
 
\end{array}
\end{equation}

 \subsection{The case of the quarter of a disk (NND) with D-cracks}\label{ss7.4}

 When $\epsilon=0$, the two first eigenvalues of this problem with Neumann on the two radii ($\theta=0,\frac  \pi 2$) are equal (due to our choice of $R_1$) 
 and correspond to $\lambda_6^D(B_1)\sim 30.47$.\\

 The third eigenvalue is either the fourth Dirichlet  eigenvalue in $B_{R_1}$ or the second $NND$ eigenvalue in the annulus. By dilation, this would be  in the first case 
  $$ (j_{0,2} /j_{0,1})^2  \lambda_2^{NND}(B_1)= (j_{0,2} /j_{0,1})^2  \lambda_4^{D}(B_1) \sim  139.05\,.$$
  We are actually in the second case, the next eigenvalues for the (NND) problem corresponding for $\epsilon=0$ to the pairs $(k,\ell)$ for the annulus: $(1,2)$, $(1,4)$, $(1,6)$ and $(2,0)$.\\

  \begin{figure}[!ht]
 \begin{center}
\includegraphics[width=12cm]{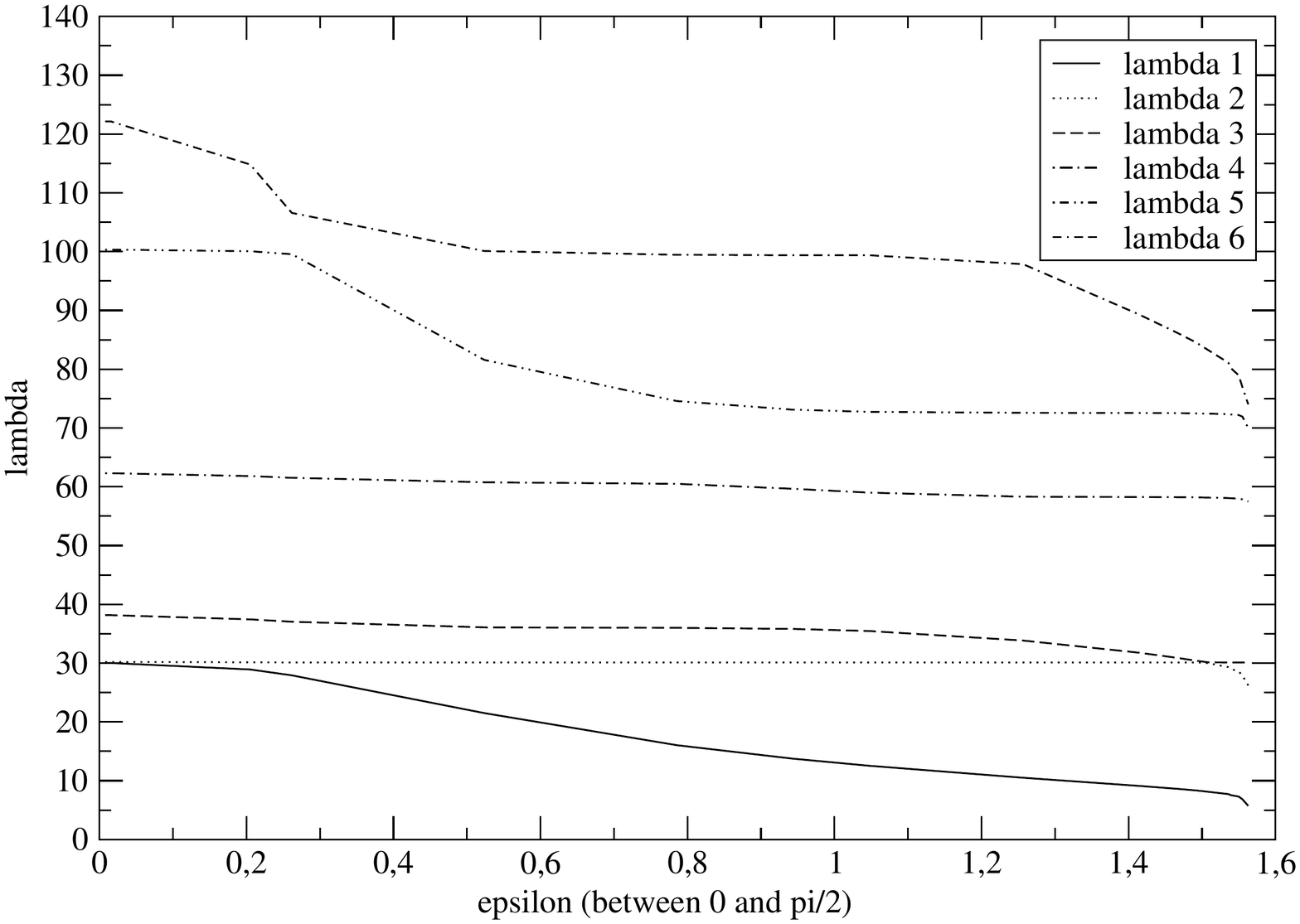}
 \caption{Case Neumann-Neumann: six first eigenvalues ($R_1=0.4356$)} 
  \end{center}
 \end{figure}

 When $\epsilon =\frac  \pi 2$, we should recover the eigenvalues of the Dirichlet problem in $B_1$ which have the right symmetry. We get:
 \begin{equation}
\begin{array}{l}
\lambda_1^{NND}(\widehat {\mathfrak D}(2,\frac  \pi 2))= \lambda_1^D( \mathfrak D(2,\frac  \pi 2))\sim 5.76  \\
\lambda_2^{NND}(\widehat {\mathfrak D}(2,\frac  \pi 2))= \lambda_4^D(\mathfrak D(2,\frac  \pi 2)) \sim 26.42 \\
\lambda_3^{NND} (\widehat {\mathfrak D}(2,\frac  \pi 2))= \lambda_6^D(\mathfrak D (2,\frac  \pi 2)) \sim 30.47\,.
\end{array}
\end{equation}
  
We recover the result predicted by our  two terms asymptotics in \eqref{eqNND}.

     \subsection{The case of the quarter of a disk (DDD) with D-cracks}

  \begin{figure}[!ht]
 \begin{center}
\includegraphics[width=12cm]{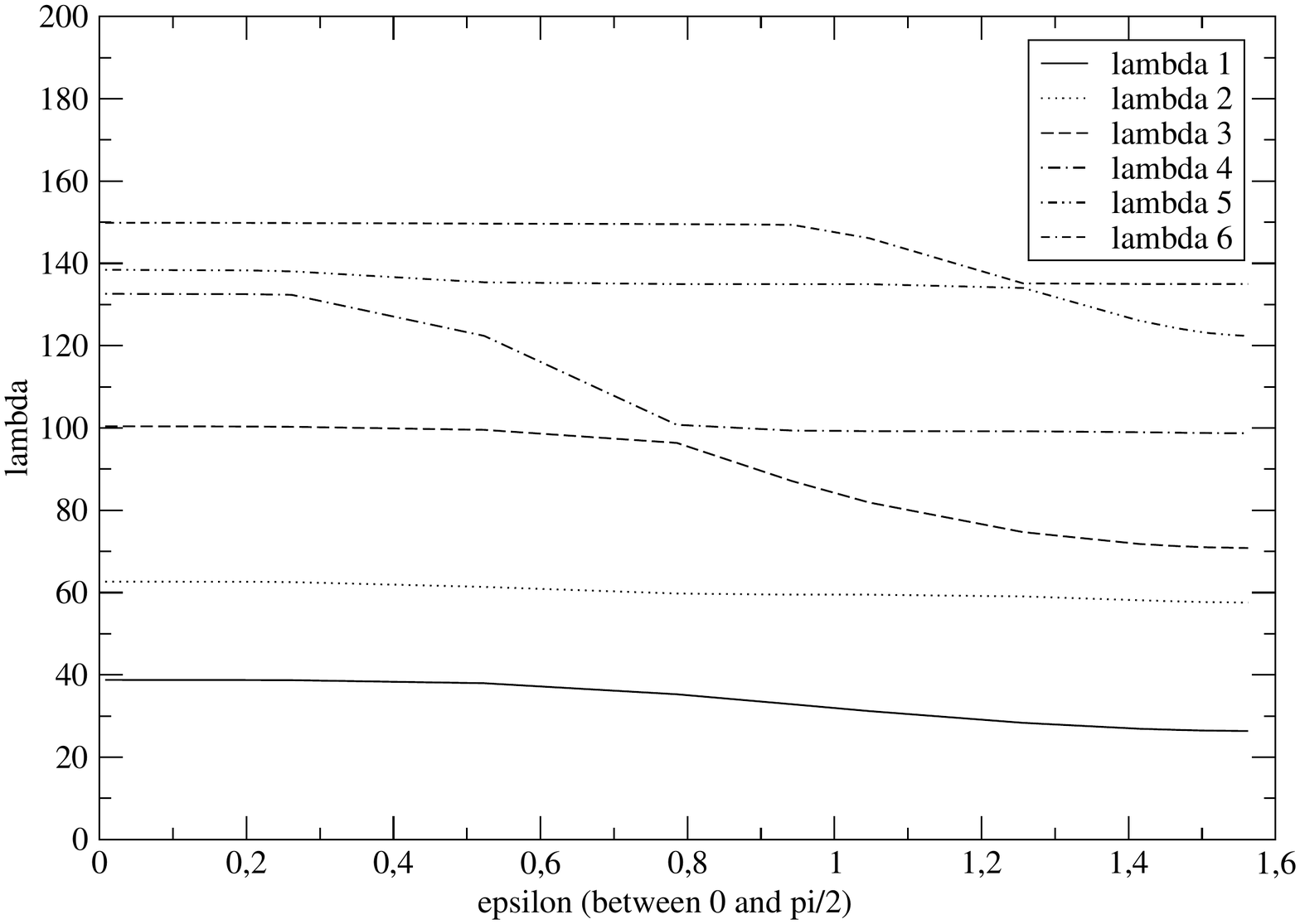}
 \caption{Case Dirichlet-Dirichlet: six first eigenvalues ($R_1=0.4356$)} 
  \end{center}
 \end{figure}

For $\epsilon =\frac \pi 2$, the first three eigenvalues  correspond to $(k,\ell)=(1, 2)$, $(k,\ell)=( 1,4)$, $(k,\ell)=(2,2)$ in the tabular of Subsection \ref{ss7.2} and correspond to the approximate values: $26.41\,;\,57.61\,;\,68.89$.\\
For $\epsilon=0$, we should recover the eigenvalue corresponding to the fourth Dirichlet eigenvalue in $B_{R_1}$ i.e. $139.05$ (which is also an eigenvalue of the $NND$ problem) with a labelling $5$. This suggests that the four first eigenvalues correspond to eigenvalues of the annulus with the right symmetry as confirmed by our computations in Subsection \ref{ss7.4}. They correspond to the pairs for the annulus $(1,2)$, $(1,4)$, $(1,6)$ and $(2,2)$. The sixth one corresponding to the pair $(1,8)$.
   
\subsection{The case of the quarter of a disk (DND) with D-cracks}
  \begin{figure}[!ht]
 \begin{center}
\includegraphics[width=12cm]{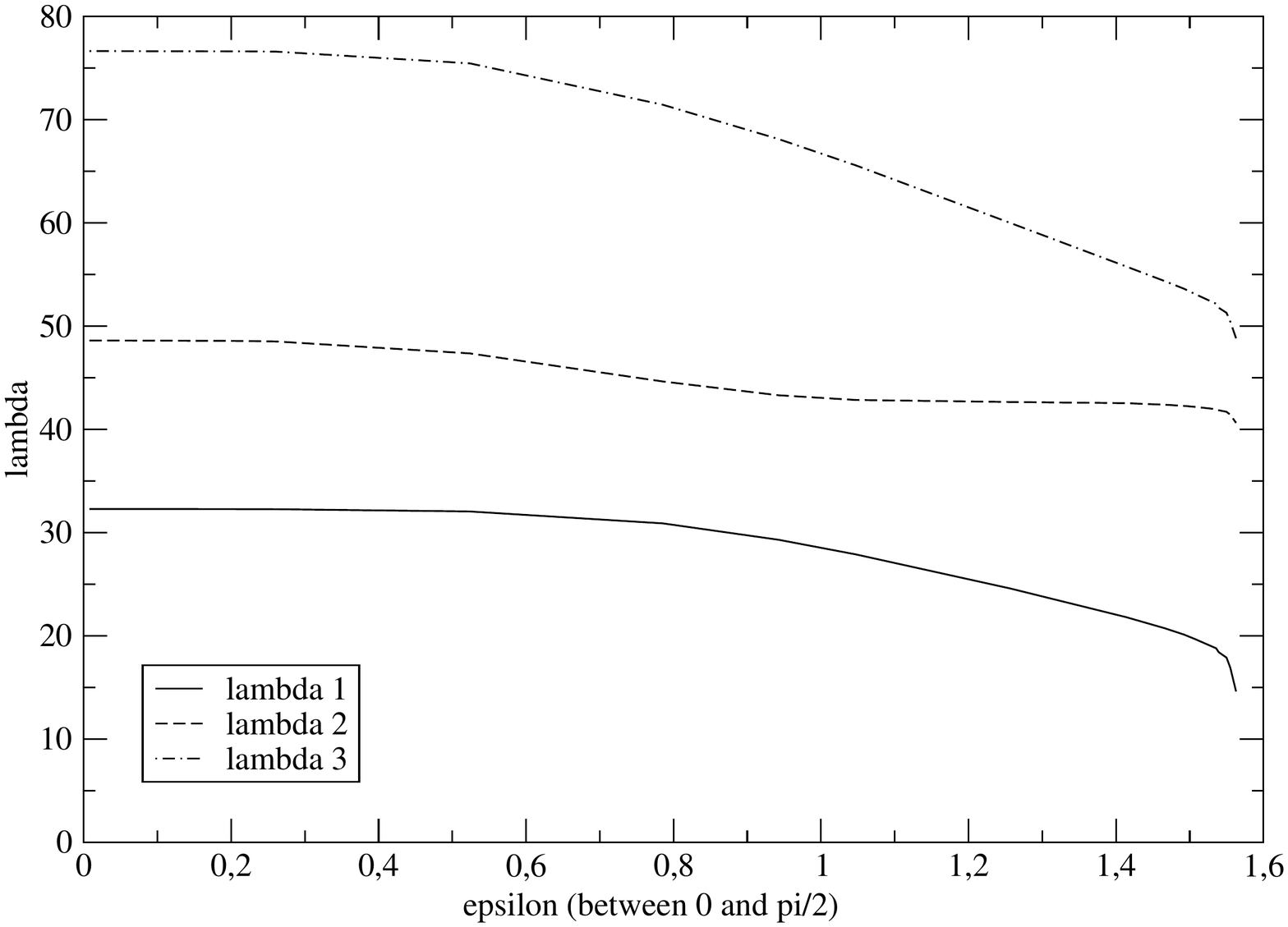}
 \caption{Case  Dirichlet-Neumann: three first eigenvalues ($R_1=0.4356$)} 
  \end{center}
 \end{figure}
 Here we keep the same $R_1, R_2$ and assume Dirichlet for $\theta=0$ and Neumann  for $\theta=\frac  \pi 2$ and we are mainly interested in the ground state energy.
 For $\epsilon=0$, the first eigenvalue is the second eigenvalue $(\ell=1)$ either in $B_{R_1}$ or  in the annulus  $B_{1}\setminus \overline{B_{R_1}}$. In the first case, 
  this would be  
    $$ (j_{0,2} /j_{0,1})^2  \lambda_2^{DND}(\widehat {\mathfrak D}(2,0))= (j_{0,2} /j_{0,1})^2  \lambda_2^{D}(B_1) \sim 77.21, $$
    which appears with labelling $3$.
  
    Hence, we have to look at the first  DND-eigenvalue of the annulus corresponding to $(k,\ell)=(1,1)$, which is approximately $32.53$. Note that the second eigenvalue
     is obtained for $(k,\ell)=(1,3)$, and is approximately $48.78$. 
    
    For $\epsilon=\frac  \pi 2$, we get as ground state 
    $$
    \begin{array}{l}
    \lambda_1^{DND} (\widehat {\mathfrak D}(2,\frac  \pi 2)) =\lambda_2^D(B_1) \sim  14.67\\
      \lambda_2^{DND} (\widehat {\mathfrak D}(2,\frac  \pi 2)) = \lambda_7^D(B_1)\sim 40.70 \\
     \lambda_3^{DND} (\widehat {\mathfrak D}(2,\frac  \pi 2)) = \lambda_9^D(B_1)\sim 49\,.
     \end{array}$$
   We also recover the behavior announced in \eqref{eqDND}.
   For $\epsilon=0$, we recover the pairs $(1,1)$, $(1,3)$ and $(1,5)$ of the annulus.

   \subsection{ The case of the quarter of a disk (NDD) with D-cracks}
 Here we keep the same pair $R_1, R_2$ and assume Neumann for $\theta=0$ and Dirichlet for $\theta=\frac  \pi 2$ and we are mainly interested in the ground state energy.

    \begin{figure}[!ht]
 \begin{center}
\includegraphics[width=12cm]{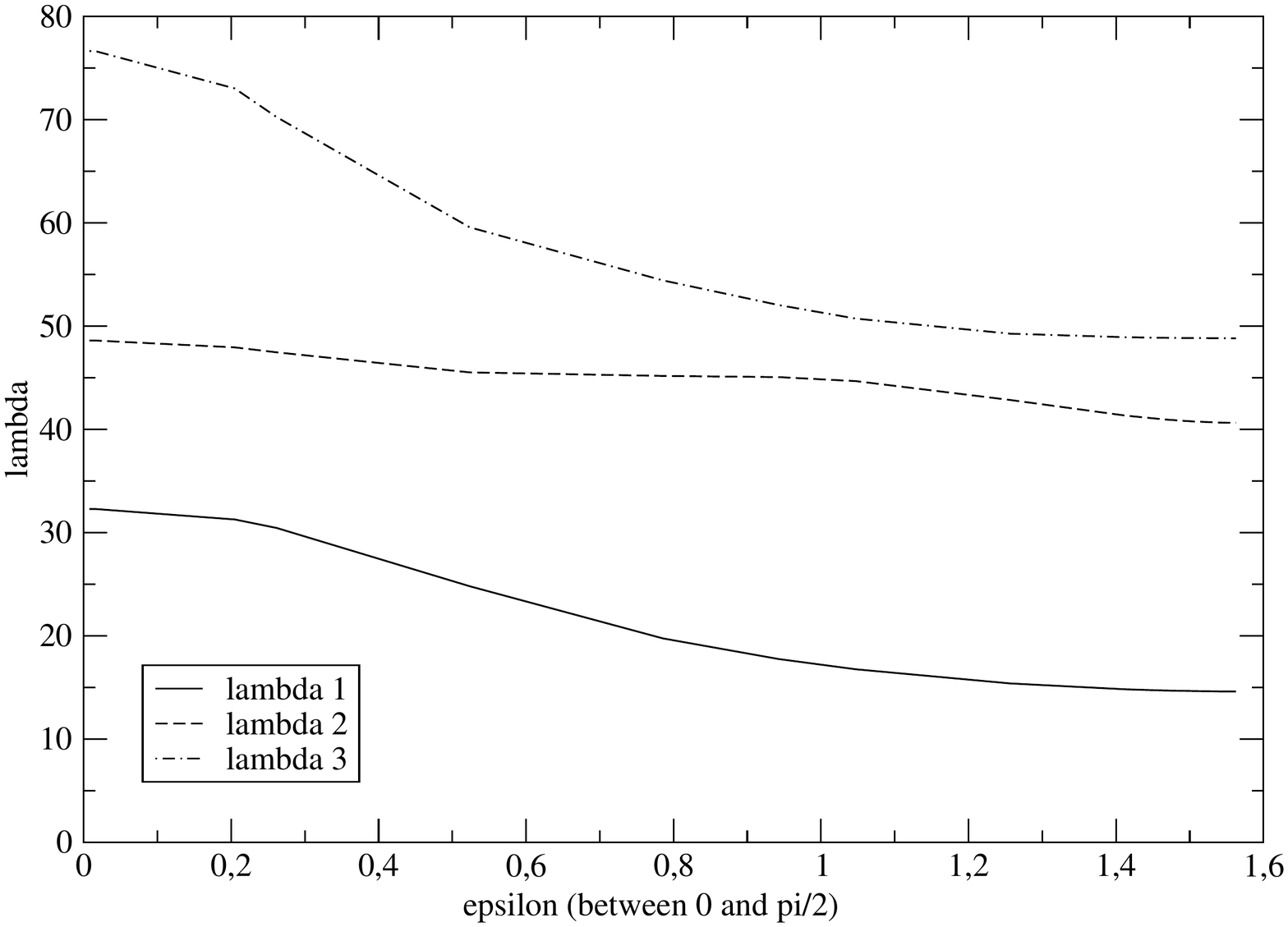}
 \caption{Case  Neumann-Dirichlet: three first eigenvalues ($R_1=0.4356$)} 
  \end{center}
 \end{figure}
For $\epsilon=0$, we recover as for (DND) the pairs $(1,1)$, $(1,3)$ and $(1,5)$ of the annulus.
 
    \subsection{Comparison between (NDD) and (DND) with D-cracks}\label{ss7.8}
 For $\epsilon=0$ and $\frac  \pi 2$ the theory says that the two spectra coincide. We recall from Section \ref{s6},  that the union of these two spectra corresponds to the odd eigenfunctions on $\mathfrak D(2,\epsilon)$ which 
  are  antisymmetric by inversion. \\
 For the ground state energies, the two curves do not cross and  have different curvature properties. This strongly suggests that they are only equal for  $\epsilon=0$ and $\frac  \pi 2$.  
 Some crossing (two points) is observed for the curves corresponding to the second eigenvalues. No crossing is observed for  the curves corresponding to the third  eigenvalues.
  \begin{figure}[!ht]
 \begin{center}
\includegraphics[width=12cm]{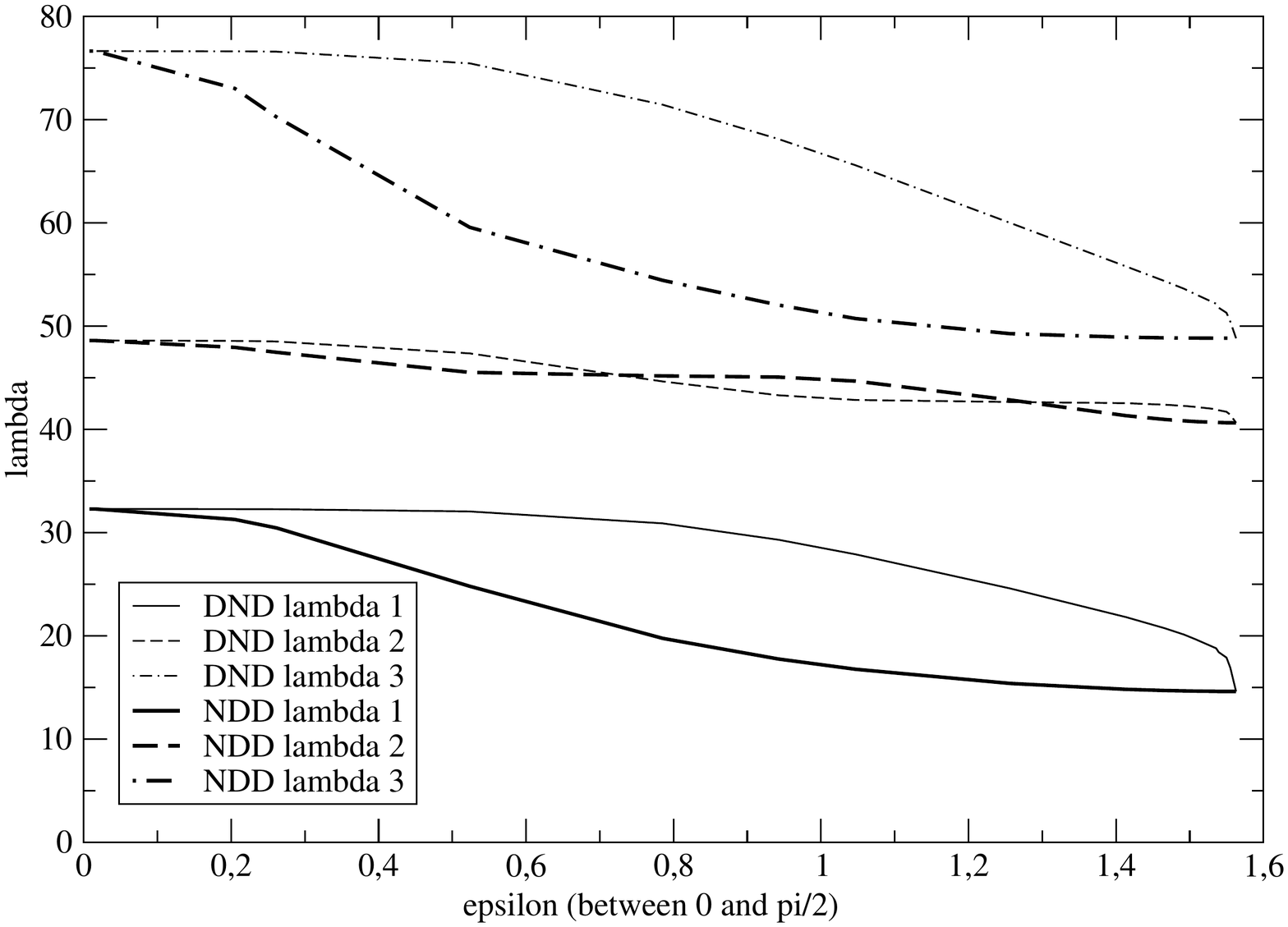}
 \caption{Case Neumann-Dirichlet and case Dirichlet-Neumann: three first eigenvalues ($R_1=0.4356$)} 
  \end{center}
 \end{figure}
 
 \subsection{The complete spectrum in $\mathfrak D(\epsilon)$.}
 Putting the whole spectrum together, we see clearly in Figure 8, as mentioned in the introduction
 of Section \ref{s6}, why there was no hope to get the multiplicity~$3$ in the case $N=2$
  by the successful approach presented for $N>2$.
 We have indeed a first crossing but it only leads to an eigenvalue of multiplicity $2$.  Look at Figure 2 corrresponding to $N=3$ for an interesting comparison.
       
        \begin{figure}[!ht]
 \begin{center}
\includegraphics[width=12cm]{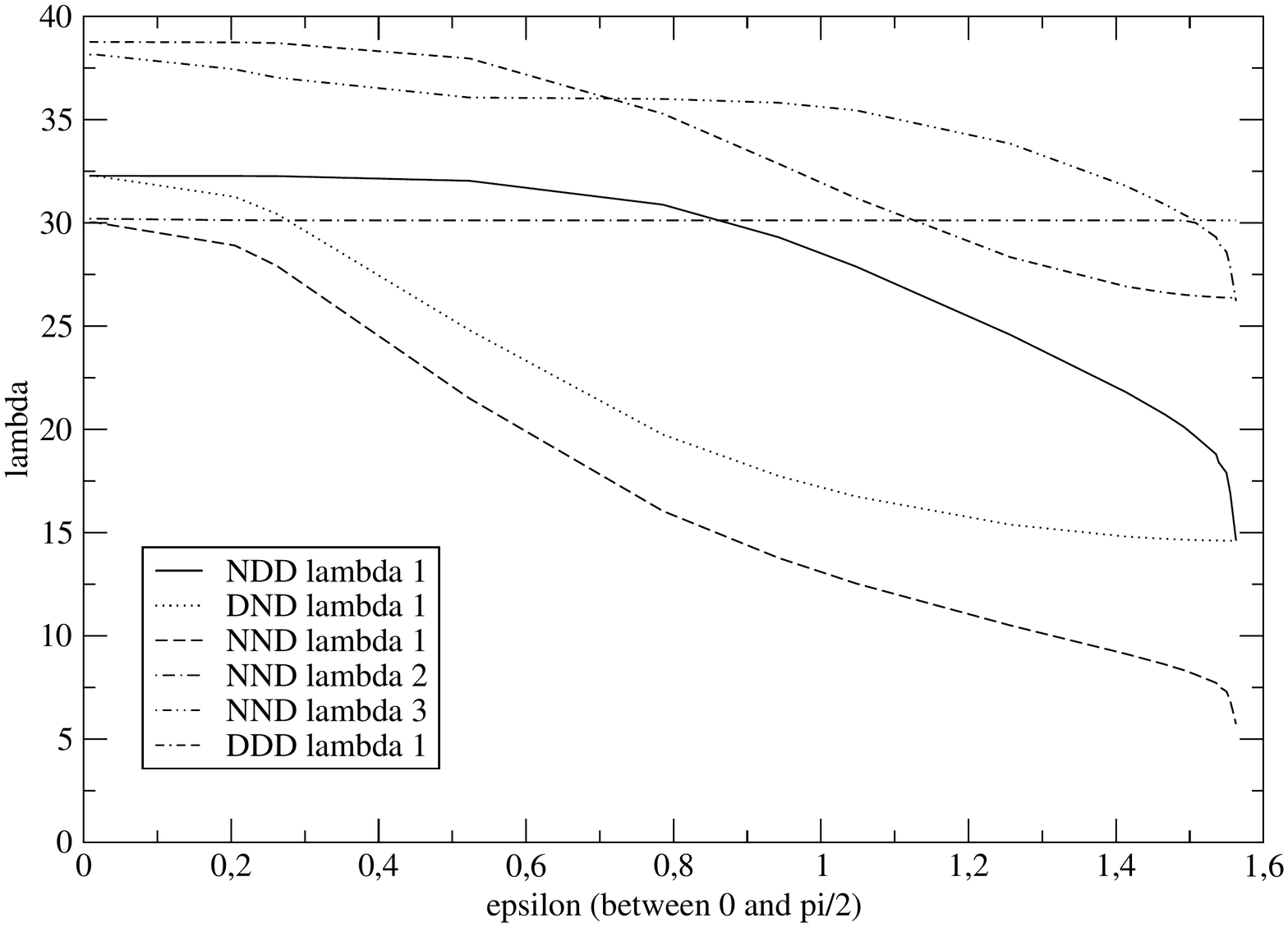}
 \caption{$N=2$, Six first eigenvalues ($R_1=0.4356$)} 
  \end{center}
 \end{figure}
 
 \subsection{On the numerical approach}\label{ss7.9}
 
Here we detail the numerical method used to obtain the different figures.  \\

We look for the numerical computation of the eigenvalue problem~:
\begin{equation}
-\Delta u (x,y)= \lambda u (x,y)
\label{eqnum-1}
\end{equation}
in the case of a domain $\Omega=\{(x,y) \in \mathbb R^2, x=r \cos(\theta), y=r \sin(\theta), r \in ]0, 1[ \mbox{ and } \theta \in ]0,\pi/2[ \}$ (a quarter of a disk). \\

In polar coordinates, Problem \eqref{eqnum-1} becomes~:
\begin{equation}
- \Delta u (r,\theta)= \lambda u (r,\theta)
\label{eqnum-2}
\end{equation}
where $\Delta u (r,\theta)= \frac {\partial^2 u}{\partial r^2} + \frac {1}{r} \frac {\partial u}{\partial r} + \frac {1}{r^2} \frac {\partial ^2 u}{\partial \theta^2}$ (singularity for $r=0$). For the boundary conditions we impose Dirichlet boundary conditions on $r=1$. Moreover, we impose Dirichlet boundary conditions on a line (D-crack) corresponding to $r=R_1$, with $0<R_1<1$, and $\theta \in [\epsilon,\pi/2]$, where $0<\epsilon<\pi/2$.\\

For the numerical discretization of the Laplacian in polar coordinates we use a second order centered finite difference scheme~:
\begin{equation}
\left(\frac {u_{i+1,j} - 2 u_{i,j} + u_{i-1,j}}{\delta r^2}\right) + \frac {1}{r_i} \left(\frac {u_{i+1,j}-u_{i-1,j}}{2 \delta r}\right) +  \frac {1}{r_i^2} \left(\frac {u_{i,j+1} - 2 u_{i,j} + u_{i,j-1}}{\delta \theta^2}\right)
\label{eqnum-3}
\end{equation}
where $u_{i,j}$ is a numerical approximation of $u(r_i,\theta_j)$ on the grid $r_i = i\,\delta r$ and $\theta_j = j\,\delta \theta$ for $i,j=1,\ldots,M-1$, where $\delta r$ and $\delta\theta$ are the steps in each direction $\delta r=\frac{R_2}{M}$ and $\delta \theta= \frac {\pi}{2M}$ ($i=0,M$ and $j=0,M$ correspond to the boundary conditions). After discretization we obtain a non symmetric tridiagonal $M\times M$  matrix $A$. For the simulations we have retained $M=180$.\\

 To compute the eigenvalues of the previous matrix $A$ we use the function DGEEV of the Lapack library. To validate the code we have considered the case of the unit disk with Dirichlet conditions, computing the six first eigenvalues to compare with \eqref{eq:7.3}. This allows us in particular the treatment of the singularity of the coefficients of the operator  appearing at $r=0$. We can also control some limits as $\epsilon \rightarrow 0$ and $\epsilon \rightarrow \frac \pi 2$ where 
 again we have theoretical values or numerical values obtained by different methods. \\

For the numerical tests we have considered $R_1=0.4356$, corresponding to an approximation of the radius of the nodal line 
of the second radial eigenfunction in $B_1$. 

\appendix
\section{Asymptotic additivity of the capacity} 
\label{appCapAdd}

We recall the definition of the condenser capacity of a compact set $K\subset \Omega$, relative to $\Omega$:
\begin{equation}
	\label{eqCapDef1}
	{\rm Cap}_\Omega(K):=\inf\left\{\int_\Omega|\nabla v|^2\,;\,v\in\Gamma_K\right\}.
\end{equation}
Here $\Gamma_K$ is the closed convex subset of $H^1_0(\Omega)$ consisting of the functions $v$ satisfying $v\ge1$ in the following sense: there exists a sequence $(v_n)$ of  functions in $H^1_0(\Omega)$ such that $v_n\ge1$ almost everywhere  in an open neighborhood of $K$ and $v_n\to v$ in $H^1_0(\Omega)$ (see for instance  Definition 3.3.19 in \cite{HP05}).  By the Projection Theorem in the Hilbert space $H_0^1(\Omega)$, there exists a unique $V_K\in\Gamma_K$ realizing the infimum, called the capacitary potential. From the minimization property, it follows immediately that  $V_K$ is harmonic in $\Omega\setminus K$. Furthermore, $V_K$ is non-negative in $\Omega$ (see for instance Item 3 of Theorem 3.3.21 in \cite{HP05}).

Let us now fix an integer $N\ge 1$, $N$ distinct points $x_1,\dots,x_N$ in $\Omega$, and $N$ families of compact subsets of $\Omega$, $(K_i^\varepsilon)_{\varepsilon>0}$ for $i\in\{1,\dots,N\}$. We assume that, for all $i\in\{1,\dots,N\}$, $(K_i^\varepsilon)_{\varepsilon>0}$ concentrates to $x_i$ as $\varepsilon\to 0$, that is to say, for any open neighborhood $U$ of $x_i$, there exists $\varepsilon_U>0$ such that $K_i^\varepsilon\subset U$ for $\varepsilon\le \varepsilon_U$. We use the notation	$K^\varepsilon:=\cup_{i=1}^N K_i^\varepsilon$. We remark that for all $i\in\{1,\dots,N\}$, $K_i^\varepsilon\subset \overline B(x_i,r_i^\varepsilon)$, with $\lim_{\varepsilon\to0}r_i^\varepsilon=0$. By monotonicity of the capacity,
\begin{equation*}
	{\rm Cap}_\Omega(K_i^\varepsilon)\le{\rm Cap}_\Omega(\overline B(x_i,r_i^\varepsilon))\to 0 \mbox{ as } \varepsilon \to 0\,.
\end{equation*}
Furthermore, by subadditivity of the capacity,
\begin{equation}
	\label{eqSub}
	{\rm Cap}_\Omega(K^\varepsilon)\le\sum_{i=1}^N{\rm Cap}_\Omega(K_i^\varepsilon)\to 0 \mbox{ as } \varepsilon \to 0\,.
\end{equation}
Let us now show that in this situation, the capacity is asymptotically additive.

\begin{proposition}
\label{propAdditivity}
	If ${\rm Cap}_\Omega(K^\varepsilon)>0$ for all $\varepsilon>0$, we have, as $\varepsilon\to0$,
	\begin{equation}
		{\rm Cap}_\Omega(K^\varepsilon)\sim\sum_{i=1}^N{\rm Cap}_\Omega(K_i^\varepsilon).
	\end{equation}
\end{proposition}

\begin{proof} Taking into account Inequality \eqref{eqSub}, we only have to prove
\begin{equation*}
	\limsup_{\varepsilon\to0}\frac{\sum_{i=1}^N{\rm Cap}_\Omega(K_i^\varepsilon)}{{\rm Cap}_\Omega(K^\varepsilon)}\le 1.
\end{equation*}
Let us set $V_\varepsilon:=V_{K^\varepsilon}$. We claim that for any (fixed)  $K$ compact subset of $\Omega\setminus\{x_1,\dots,x_N\}$,  $V_\varepsilon $ converges to $0$ as $\varepsilon$ tends to $0$, uniformly in $K$. Indeed, for $\varepsilon$ small enough, $K\subset\Omega\setminus K^\varepsilon$, 
so that $V_\varepsilon$ is harmonic in an open neighborhood $U$  of $K$. Let us fix $r>0$ such that $\overline B(x,r)\subset U$ for all $x\in K$. From the Mean Value Formula, for all $x\in K$,
\begin{equation*}
	0\le V_\varepsilon(x)=\frac1{\omega_dr^d}\int_{B(x,r)}V_\varepsilon 
	\le 	\frac1{(\omega_dr^d)^\frac12}\|V_\varepsilon\|_{L^2(\Omega)}.
\end{equation*}
The claim follows from the fact that $V_\varepsilon$ tends to $0$ in $L^2(\Omega)$. 

Let us now fix $R>0$ such that the closed balls $\overline B_i:=\overline B(x_i,R)$ are contained in $\Omega$ and mutually disjoint. By the above claim, $\delta_\varepsilon\to0$ when $\varepsilon \to 0$, where
\begin{equation*}
	\delta_\varepsilon:=\max_{i\in\{1,\dots,N\}}\max_{\partial \overline B_i} V_\varepsilon.
\end{equation*}
For $i\in\{1,\dots,N\}$, we define
\begin{equation*}
	v_i^\varepsilon:=\frac1{1-\delta_\varepsilon}\left(V_\varepsilon-\delta_\varepsilon\right)_+\mathbf 1_{\overline B_i}.
\end{equation*}
We have $v_i^\varepsilon\in H^1_0(\Omega)$, and furthermore $v_i^\varepsilon\in \Gamma_{K_i^\varepsilon}$. Indeed, let us pick a sequence $(\varphi_n)$ converging in $H^1_0(\Omega)$ to $V_\varepsilon$ and such that, for all $n$, $\varphi_n\ge 1$ almost everywhere in a neighborhood of $K_i^\varepsilon$. Setting 
\begin{equation*}
	\psi_n:=\frac1{1-\delta_\varepsilon}\left(\varphi_n-\delta_\varepsilon\right)_+\mathbf 1_{\overline B_i},
\end{equation*}  
we get $\psi_n\in H^1_0(\Omega)$ for all $n$ and $(\psi_n)$ converges to $v_i^\varepsilon$ in $H^1_0(\Omega)$. Furthermore, for all $n$, $\varphi_n(x)\ge1$ implies $\psi_n(x)\ge1$, so that $\psi_n\ge1$ almost everywhere in a neighborhood of
$K_i^\varepsilon$. 
It follows that
\begin{equation*}
	{\rm Cap}_\Omega(K_i^\varepsilon)\le \int_\Omega \left|\nabla v_i^\varepsilon \right|^2.
\end{equation*} 
Summing for $i$ ranging from $1$ to $N$, we find
$$
\begin{array}{ll}
	\sum_{i=1}^N{\rm Cap}_\Omega(K_i^\varepsilon)& \le \sum_{i=1}^N\int_\Omega \left|\nabla v_i^\varepsilon \right|^2\\
	& \le \frac1{\left(1-\delta_\varepsilon\right)^2}\int_\Omega\left|\nabla V_\varepsilon\right|^2\\ & =\frac1{\left(1-\delta_\varepsilon\right)^2}{\rm Cap}_\Omega(K^\varepsilon).  \qedhere
\end{array} 
$$
\end{proof}
~\\
{\bf Acknowledgements.}\\
 B.H and T.H-O would like to thank  S. Fournais for many discussions on the subject along the years. B.H and  C.L would like to thank  the Mittag-Leffler institute where part of this work was achieved. C.L  was partially supported by the Swedish Research Council (Grant D0497301). \\

\bibliographystyle{plain}

\end{document}